\documentclass{ifacconf}

\usepackage{graphicx}      
\usepackage{natbib}        

\usepackage{amsfonts}
\usepackage{amssymb}
\usepackage{enumitem}
\usepackage{amsmath}
\usepackage{mathtools}
\usepackage{color}

\newcommand{\R}{\mathbb{R}}
\newcommand{\N}{\mathbb{N}}
\newcommand{\Rn}{\R^n}

\newcommand{\II}{\mathcal{I}}
\newcommand{\JJ}{\mathcal{J}}
\newcommand{\GG}{\mathcal{G}}
\newcommand{\VV}{\mathcal{V}}
\newcommand{\EE}{\mathcal{E}}
\newcommand{\UU}{\mathcal{U}}
\newcommand{\WW}{\mathcal{W}}

\DeclareMathOperator{\distG}{\operatorname{d_\mathcal{G}}} 

\begin{document}
\begin{frontmatter}

\title{Decaying Sensitivity of the Zero Solution for a Class of Nonlinear Optimal Control Problems} 

\thanks[footnoteinfo]{The authors thank Andrii Mironchenko, Manuel Schaller, and Karl Worthmann for valuable discussions and helpful comments. This work was funded by the Deutsche Forschungsgemeinschaft (DFG, German Research Foundation) - project number 463912816.}

\author[First]{Lars Grüne} 
\author[First]{Mario Sperl}

\address[First]{Mathematical Institute, University of Bayreuth, Bayreuth, Germany (e-mail: lars.gruene@uni-bayreuth.de, mario.sperl@uni-bayreuth.de).}

\begin{abstract}                
We study spatial decay properties of sensitivities in a nonlinear optimal control problem with a graph-structured interaction topology. For a problem with nonlinear decoupled dynamics and quadratic cost, we show that a perturbation of the zero initial condition at a single node induces an optimal trajectory whose node-wise norms decay exponentially with the graph distance from the perturbed node. The analysis, based on a nonlinear null-controllability condition, provides a first step toward extending known spatial decay results from linear–quadratic to nonlinear systems. A numerical example illustrates the theoretical findings.
\end{abstract}

\begin{keyword}
 optimal control, interconnected nonlinear systems, decaying sensitivity,  numerical methods for optimal control, large-scale systems
\end{keyword}
\end{frontmatter}

\section{Introduction}
High-dimensional optimal control problems arise naturally in a wide range of applications, including large-scale networked systems, multi-agent coordination, traffic or energy systems, and distributed robotics. In such settings, the system consists of a large number of interacting subsystems or agents, and the resulting state space grows proportionally with the number of components. As a consequence, classical centralized solution approaches quickly become computationally infeasible, a phenomenon commonly referred to as the curse of dimensionality (\cite{Bellman1957}). 

To address this challenge, significant effort has been devoted to the development of decentralized and distributed control strategies. A key idea underlying these approaches is that, in many large-scale systems, interactions between subsystems exhibit a localized structure: individual agents are directly influenced only by a limited subset of other agents, for instance, through spatial proximity or network connections. Exploiting such locality properties is essential for reducing computational complexity and enabling scalable control designs. This raises the fundamental question of how perturbations or changes in the state of one subsystem affect the optimal behavior of other subsystems. In particular, understanding the sensitivity of optimal trajectories and value functions with respect to localized perturbations is crucial for justifying decentralized approximations and localized solution methods. If the influence of a perturbation decays sufficiently fast with respect to an appropriate notion of distance between subsystems, then the global optimal control problem may be well approximated by considering only localized subproblems.

Motivated by these considerations, recent research has focused on sensitivity analysis and spatial decay properties in large-scale optimal control problems. In \cite{sperl2023separable,sperl2026separable}, a separable approximation of the optimal value function is constructed based on a spatial decay property of sensitivities between subsystems. This property enables an efficient representation of the optimal value function using separable structured neural networks. A crucial ingredient for the scalability of the approach is that the sensitivity decay holds uniformly with respect to the system dimension. A natural objective is therefore to identify control-theoretic conditions that guarantee uniform spatially decaying sensitivity.

Building on earlier work on spatial decay phenomena in graph-structured nonlinear optimization problems (see \cite{shin2022expDecayNLP}), the paper \cite{shin2022lqr} establishes an exponential decay of the optimal feedback matrix for time-discrete linear–quadratic optimal control problems under stabilizability and detectability assumptions. Complementary to this line of work, \cite{zhang2023optimal} shows that the entries of the optimal feedback matrix decay exponentially, assuming that the system matrices exhibit a corresponding decay structure. We emphasize that all these results require linearity of the dynamics.

Further notions of decay and sensitivity have been investigated in other contexts. Spatial decay of sensitivities has also been studied for optimal control problems governed by partial differential equations and more general evolution equations, see \cite{goettlich2026spatialexponentialdecayperturbations,goettlich2025perturbationspdeconstrainedoptimalcontrol,oppeneiger2025spatialdecayperturbationshyperbolic}. Exponential decay with respect to time, rather than space, is for example analyzed in \cite{shin2021controllability,na2020exponential,GrSS20}. Moreover, scalable reinforcement learning methods exploiting exponential decay properties of the $Q$-function are proposed in \cite{qu2022scalableReinforcement}.

In this work, we establish spatial decay properties of sensitivities for a class of finite-dimensional nonlinear optimal control problems under an exponential null-controllability assumption. In particular, we investigate how the effect of a nonzero initial value in a single subsystem propagates through a network of coupled subsystems, and show that this effect decays with their distance in the network. Our result can be viewed as a first step toward a nonlinear extension of the sensitivity decay results for linear–quadratic optimal control problems obtained in \cite{shin2022lqr}. However, the present result is not a straightforward generalization of the linear–quadratic setting since the analysis therein relies crucially on matrix-based arguments that are no longer available in the nonlinear case. The nonlinear nature of the problem introduces fundamentally new difficulties, which require a different proof strategy. In particular, in order to derive decay properties without resorting to localization arguments, we introduce a null-controllability condition tailored to the nonlinear setting. To focus on this core mechanism, we consider a simplified yet nontrivial class of problems with nonlinear but decoupled dynamics and a quadratic cost functional. 

The outline of the paper is as follows. After introducing the problem setting in the next section, Section~\ref{sec:main_result} presents the main theorem. Its proof is given in Section~\ref{sec:proof}. A numerical test case is discussed in Section~\ref{sec:numerics}, and the paper concludes with Section~\ref{sec:conclusion}.

\section{Setting} \label{sec:setting}

\subsection{Dynamics and Norm Conventions}

We consider decoupled nonlinear dynamics  

\begin{equation}
    \dot x_i = f_i(x_i,u_i) 
    \label{eq:sysi}
\end{equation}
with $i=1,\ldots,s$, where the vector fields $f_i \colon \R^{n_i}\times\R^{m_i}\to\R^{n_i}$  satisfy $f_i(0,0)=0$ for all $i$. The initial conditions at time $t=0$ are denoted by $x_{0,i}$. The aggregated state and control vectors are given by  
\[
x = \left(\begin{array}{c}x_1\\ \vdots\\ x_s \end{array} \right) \in\R^n,
u = \left(\begin{array}{c}u_1\\ \vdots\\ u_s \end{array} \right) \in\R^m,
\]
where $n=\sum_{i=1}^s n_i$, and $m = \sum_{i=1}^s m_i$. 

We consider a set of initial values $\Omega \subseteq \Rn$ containing a neighborhood of the origin and denote by $\UU$ the set of admissible controls, given as a suitable space of measurable functions $u : [0,\infty) \to U$ with $U \subseteq \mathbb{R}^m$. Furthermore, we assume that the system \eqref{eq:sysi} is forward complete, i.e., for any $x_0 \in \Omega$ and any $u \in \UU$ the solution $x(\cdot,x_0,u)$ of \eqref{eq:sysi} exists and is unique for all nonnegative times.

For an index set $\II \subseteq \{ 1, \dots, s\}$ we denote by $x_\II \in \R^{n_\II}$ and $u_\II \in \R^{m_\II}$ the vectors that comprise the states and control from all subsystems in $\II$, respectively, with $n_\II := \sum_{i \in \II} n_i$ and $m_\II := \sum_{i \in \II} m_i$. 
We endow $\R^n$ with the norm
\[
\|x\| = \biggl(\sum_{i=1}^s |x_i|^2 \biggr)^{1/2},
\]
where each $|\cdot|$ denotes a fixed (but otherwise arbitrary) norm on $\mathbb R^{n_i}$.  For any index set $\II \subseteq \{ 1, \dots, s\}$ we use the same convention for the subvectors $x_{\II}$ of $x$, i.e.,
\[
\|x_{\II}\| = \biggl(\sum_{i\in\II} |x_i|^2 \biggr)^{1/2}.
\]
This norm satisfies two elementary properties that will be used repeatedly in the following. First, for any function $x \colon \R_{\geq 0} \to \Rn$, its squared $L_2$-norm decomposes componentwise:
\begin{align*}
    \begin{split} 
        \|x\|_{L_2}^2 & = \int_0^\infty \|x(t)\|^2 dt = \int_0^\infty \sum_{i=1}^s |x_i(t)|^2 dt \\ 
        & = \sum_{i=1}^s \int_0^\infty |x_i(t)|^2 dt = \sum_{i=1}^s \|x_i\|_{L_2}^2. 
    \end{split}
\end{align*}

Second, if $\II_1 \subseteq \II_2 \subseteq \{1, \dots, s \}$ are index sets, then monotonicity with respect to the index set holds:
\[
\|x_{\II_1}\| \le \|x_{\II_2}\|.
\]
The same monotonicity property applies to the associated $L_2$-norms of time-dependent functions.

\subsection{Optimal Control Problem}
The subsystems \eqref{eq:sysi} are coupled via quadratic costs
\begin{equation}
    \ell(x,u) = \sum_{i,j=1}^s x_i^T Q_{ij} x_j + \sum_{k=1}^s u_k^T R_k u_k = x^TQx + u^TRu, 
\label{eq:ell}
\end{equation} 
where $Q$ and $R$ are assumed to be symmetric and positive definite. Hence, there exist constants $\mu, M_Q, M_R > 0$ such that for all $x \in \Rn$ 
\begin{equation} \label{eq:Q_R_bounds}
    \mu\|x\|^2 \le x^TQ x, \quad \|Qx\| \le M_Q \|x\|, \quad \|Rx\| \le M_R\|x\|.
\end{equation}
We consider the infinite horizon optimal control problem of minimizing the cost functional
\begin{equation}
    J(x_0,u) = \int_0^\infty \ell\big(x(t,x_0,u),u(t)\big) dt 
    \label{eq:J}
\end{equation}
for a given initial value $x_0 \in \Omega$ over all admissible controls $u \in \UU$. We assume that for every $x_0 \in \Omega$ there exists an optimal control $u^* \in \UU$ such that 
\begin{equation*}
    J(x_0, u^*) = \inf_{u \in \UU} J(x_0, u).
\end{equation*}
If the optimal control is not unique, $u^*$ is understood to denote a fixed representative. We denote the corresponding optimal trajectory via $x^*(\cdot) = x(\cdot, x_0, u^*)$.
If this assumption is not satisfied, we expect that the following results can be extended to nearly optimal trajectories; however, we leave a rigorous analysis of this case for future work.

While the dynamics of the subsystems in \eqref{eq:sysi} are decoupled, the overall optimal control problem with cost functional $J$ defined in \eqref{eq:J} introduces a coupling through the matrix $Q$ in \eqref{eq:ell}. We represent this interconnection of the subsystems by an undirected graph $\GG = (\VV, \EE)$, where each subsystem corresponds to a node identified with its index, i.e., we set $\VV = \{1, \dots, s\}$. For each pair of distinct nodes $i \neq j$, the unordered pair $\{i, j\}$ is an edge in the graph if and only if $Q_{ij} \neq 0$ (which, by the symmetry of $Q$, is equivalent to $Q_{ji} \neq 0$). The graph distance between two nodes is defined as the length of the shortest path connecting them and is denoted by $\distG(i,j)$. Note that $\distG(i,j) \ge 2$ for $i \neq j$ if and only if $Q_{ij} = 0$. For a nonempty subset $\WW \subseteq \VV$ and a node $i \in \VV$, we define $\distG(i,\WW) = \min_{j\in \WW} \distG(i,j)$. In the following, we refer to the setting introduced in this subsection as an optimal control problem (OCP) of the form \eqref{eq:sysi} - \eqref{eq:J}.
 
\subsection{Problem Formulation}
Note that from the assumptions on $f_i$, $Q$, and $R$ it immediately follows that $x^*(t)\equiv 0$ with control $u^*(t)\equiv 0$ is the optimal solution for the initial condition $x_0=0$. In this paper, we investigate how the solution changes across the different subsystems if a single subsystem has a nonzero initial value $x_{0,i^*}\neq 0$ while all other initial values remain $0$. Our aim is to quantify how this influence propagates through the network, with an estimate that reflects the distance of each subsystem from $i^*$ and remains independent of the overall dimension, i.e., independent of $s$ and $n_i$, and thus also independent of $n$.

\section{Main Result} \label{sec:main_result}
To formulate our main result on sensitivity decay, we impose the following assumption of exponential null-controllability for a system of equations \eqref{eq:sysi}. 

\begin{assum}
    There are $C,\sigma>0$ such that for each $x_0\in\Omega$ there is a control $u_{x_0} \in \UU$ such that for all $t \geq 0$ the inequalities
    \begin{equation}
        \|x(t,x_0,u_{x_0})\| \le Ce^{-\sigma t}\|x_0\|, \; \|u_{x_0}(t)\| \le Ce^{-\sigma t}\|x_0\|
        \label{eq:asumm_null_contr}
    \end{equation}
    hold.
\label{asm:expstab}
\end{assum}

Assumption \ref{asm:expstab} ensures that for every initial value $x_0 \in \Omega$ there exists an exponentially bounded control that steers the trajectory exponentially fast to the origin. In particular, any asymptotically controllable system  \eqref{eq:sysi} on a compact domain $\Omega$, which has a stabilizable linearization at $0$ satisfies Assumption \ref{asm:expstab}. Given this assumption, the main result guarantees that the sensitivity of the zero solution decreases exponentially as a function of the distance in the underlying graph of subsystems. In the following, for $a \in \mathbb{R}$, we write $\lceil a \rceil$ for the smallest integer greater than or equal to $a$.

\begin{thm} Consider an OCP of the form \eqref{eq:sysi}-\eqref{eq:J} satisfying Assumption \ref{asm:expstab}. Pick any $i^* \in \VV$ and let the initial value $x_0 \in \Omega$ be such that $x_{0,i^*}\ne 0$ and $x_{0,i}=0$ for all $i\ne i^*$. Denote the optimal trajectory starting at $x_0$ by $x^*$. Then for any subset $\WW\subseteq \VV$, it holds that
\begin{equation} \label{eq:main_inequ}
    \|x_\WW^*\|_{L_2} \le S \rho^{\distG(i^*,\WW)}|x_{0,i^*}|,
\end{equation}
where $S=2\max\{C_{\textrm{init}},C_{\textrm{prop}}\}$, $\rho = 2^{-1/q}<1$, $q=\lceil S \rceil^2$, 
\begin{equation*}
    C_{\textrm{init}} =\sqrt{\frac{M_Q+M_R}{2\sigma\mu}}C, \quad  C_{\textrm{prop}} =2 M_Q \mu^{-1}, 
\end{equation*}
and $\mu, M_Q, M_R > 0$ are chosen to satisfy the matrix-inequalities \eqref{eq:Q_R_bounds}, i.e., 
\begin{equation*}
     \mu\|x\|^2 \le x^TQ x, \quad \|Qx\| \le M_Q \|x\|, \quad \|Rx\| \le M_R\|x\|. 
\end{equation*}
\label{thm:main}
\end{thm}

Observe that for the particular case where $\WW = \{ j \}$ for some $j \in \VV$, Theorem \ref{thm:main} states that the norm of the solution in subsystem $j$ diminishes exponentially with respect to the distance of the nodes $i^*$ and $j$ in the graph. In other words, distant subsystems are affected only weakly by a local perturbation of the initial condition and the effect of a perturbation is concentrated near its origin.

\begin{rem}[Dimension-Independent Bounds]
For the application of Theorem~\ref{thm:main} to high-dimensional optimal control problems, it is desirable that the constants in~\eqref{eq:main_inequ} are independent of the state-space dimension. To make this precise, consider a family of optimal control problems of increasing dimension for which the constants in Theorem~\ref{thm:main} can be chosen uniformly across the family. By inspecting the dependence of the constants $S$ and $\rho$, this requires $C$ and $\sigma$ from Assumption \ref{asm:expstab} as well $\mu$, $M_Q$, and $M_R$ in~\eqref{eq:Q_R_bounds} to be uniformly bounded. 
Regarding $C$ and $\sigma$ from Assumption~\ref{asm:expstab}, in the present setting of decoupled dynamics, uniformity is naturally satisfied when the dimension increases by adding subsystems from a fixed class of agents with identical dynamics. Concerning the constants $M_R$ and $M_Q$, structural properties such as block-diagonality of $R$ and sparsity of $Q$ are beneficial, but additional assumptions on the growth of the blocks and couplings across the family are required to ensure uniform bounds.
Establishing general conditions guaranteeing such uniformity properties, as well as relaxing these requirements, is left for future work.
\end{rem}

\section{Proof of the main result} \label{sec:proof}
In this section, we prove Theorem \ref{thm:main}. Throughout, we consider an optimal control problem of the form \eqref{eq:sysi}–\eqref{eq:J}, and all statements are understood with respect to this setting. The proof is based on two auxiliary bounds on the norms of optimal trajectories, which are established in Subsections \ref{subsec:initial_pert} and \ref{subsec:prop_pert} and subsequently combined in Subsection \ref{subsec:proof_main}. 

\subsection{Influence of the initial perturbation} \label{subsec:initial_pert}
We begin by estimating how a perturbation of the initial value $x_0 = 0$ in one component, i.e. setting $x_{0,i^*} \neq 0$ for some $i^* \in \VV$, affects the norm of the optimal trajectory of the overall system. 

\begin{lem}
    Let $i^* \in \VV$. Consider an initial value $x_0 \in \Omega$ with $x_{0,i^*}\ne 0$ and $x_{0,i}=0$ for all $i\ne i^*$ and denote the optimal trajectory starting at $x_0$ by $x^*$. Let Assumption \ref{asm:expstab} hold. Then the inequality
    \[ \|x^*\|_{L_2} \le C_{\textrm{init}} |x_{0,i^*}| \]
    holds with $C_{\textrm{init}} =\sqrt{\frac{M_Q+M_R}{2\sigma\mu}}C$.
\label{lemma:inipert}
\end{lem}

\begin{pf}
    Observe that $\|x_0\| = |x_{0,i^*}|$, as all other components of $x_0$ equal $0$. Hence, from Assumption \ref{asm:expstab} we know that there exist $C, \sigma > 0$ and $u_{x_0} \in \UU$ such that for all $t \geq 0$ the estimates in \eqref{eq:asumm_null_contr} hold. This implies that
    \begin{align*}
    & J(x_0,u_{x_0}) = \int_0^\infty \ell(x(t,x_0,u_{x_0}),u_{x_0}(t)) dt \\ 
    & \le \int_0^\infty(M_Q+M_R)C^2e^{-2\sigma t}|x_{0,i^*}|^2 dt\\
    & = \frac{(M_Q+M_R)C^2}{2\sigma} |x_{0,i^*}|^2.
\end{align*} 
    Since the optimal control $u^*$ generating the optimal trajectory $x^*$ minimizes $J$, it follows that
    \[ 
    J(x_0,u^*) 
    \! = \! \int_0^\infty \! \ell(x^*(t),u^*(t)) dt 
    \le \frac{(M_Q+M_R)C^2}{2\sigma} |x_{0,i^*}|^2.
    \]
    Moreover, by \eqref{eq:Q_R_bounds} we have that
    \begin{align*}
        \|x^*\|_{L_2}^2  & = \int_0^\infty \|x^*(t)\|^2 dt \\
        & \le \frac1\mu \int_0^\infty x^*(t)^T Q x^*(t) dt \le \frac1\mu J(x_0,u^*).
    \end{align*}
    The combination of these inequalities yields the assertion after taking the square root.
    \hfill\qed
\end{pf}

Note that the proof of Lemma~\ref{lemma:inipert} does not rely on the decoupled structure of the considered dynamics in \eqref{eq:sysi}; it only requires the matrix bounds in~\eqref{eq:Q_R_bounds} and the null-controllability condition in Assumption~\ref{asm:expstab}. The particular structure of the dynamics will be needed in the following subsection.

\subsection{Propagation of error perturbation} \label{subsec:prop_pert}
Lemma \ref{lemma:inipert} has shown how an initial perturbation affects the optimal solution. In order to see how the resulting perturbation propagates through the graph, we need a second lemma, which we provide in this subsection. 

Consider two index sets $\II$, $\JJ \subseteq \VV$, $\II \cap \JJ = \emptyset$, and fix some measurable functions $\hat x_j \colon \R_{\geq 0} \to \R^{n_j}$ for all nodes $j \in \JJ$. These functions may be interpreted as externally prescribed (or perturbed) trajectories, and our goal is to quantify how such perturbations influence the optimal 
trajectories associated with the nodes in $\II$. To this end, we define
\begin{align*}
    \ell_\II(x_\II, u_\II, \hat x_\JJ) =  \underbrace{\sum_{i\in\II, j\in \JJ} (x_i^T Q_{ij} \hat x_j + \hat x_j^T Q_{ji} x_i)}_{=: \ell_1(x_\II, \hat x_\JJ)} \\
    + \underbrace{\sum_{i, i'\in\II} x_i^T Q_{i i'} x_{i'}}_{=:\ell_2(x_\II)}
    + \underbrace{\sum_{i\in\II} u_i^T R_i u_i}_{=:\ell_3(u_\II)}
\end{align*}
and denote by $J_\II$ the corresponding infinite horizon cost functional
\begin{equation*}
    J_\II(x_{\II,0},u_\II,\hat x_\JJ) = \int_0^\infty \ell_\II(x_\II(t, x_{\II,0}, u_\II), u_\II(t), \hat x_\JJ(t)) dt. 
\end{equation*}
The resulting reduced OCP then takes the form 
    \begin{align}
        \begin{split} \label{eq:ocp_II}
            \min_{u_\II} \, & J_\II(x_{\II,0},u_\II,\hat x_\JJ), \\ 
            \text{s.t. } & \dot x_i = f_i(x_i, u_i), \, i \in \II.
        \end{split}
\end{align}
The following Lemma estimates the norm of the optimal solution of $\eqref{eq:ocp_II}$ in dependence of the norm of the fixed functions $\hat x_\JJ$. 

\begin{lem}
    Let $\II, \JJ \subseteq \VV$ be disjoint index sets. For fixed functions $\hat x_\JJ$, consider the reduced OCP \eqref{eq:ocp_II} with initial condition $x_{\II,0} = 0$. Then the optimal solution $\tilde x^*_\II$ of \eqref{eq:ocp_II} satisfies
    \begin{equation*}
        \| \tilde x^*_\II\|_{L_2} \le C_{\text{prop}} \| \hat x_\JJ\|_{L_2}
    \end{equation*}
    with $ C_{\textrm{prop}} =2 M_Q \mu^{-1}$. 
\label{lemma:pertprop}
\end{lem}

\begin{pf}
Inserting $u_\II^0\equiv 0$, we obtain that $x_\II(t,x_{\II,0},u_\II^0)\equiv 0$ and thus by the definition of $\ell_\II$
\[ J_\II(x_{\II,0},u_\II^0,\hat x_\JJ) = 0. \]
Denoting the optimal control for \eqref{eq:ocp_II} with initial value $0$ by $\tilde u_\II^*$, this implies that 
\begin{equation}
    J_\II(x_{\II,0},\tilde u_\II^*,\hat x_\JJ) \le 0.
\label{eq:Jle0}\end{equation} 
Since $R$ is positive definite, we have $\ell_3(\tilde u_\II^*(t))\ge 0$ for all $t$. Thus, from \eqref{eq:Jle0} we can conclude that
\begin{equation}
    \int_0^\infty \ell_1(\tilde x_\II^*(t), \hat x_\JJ(t)) + \ell_2(\tilde x_\II^*(t)) dt \le 0.
\label{eq:intell12}\end{equation}
Next we observe that if we take a vector $x_\II\in\R^{n_\II}$ and extend it to a vector $x\in\R^n$ by setting $x_k = 0$ for all $k \in \VV \setminus \II$, we obtain that 
\begin{equation*}
    \ell_2(x_\II) = \ell(x,0) \ge \mu\|x\|^2 = \mu\|x_\II\|^2.
\end{equation*}
By extending $\hat x_\JJ\in\R^{n_\JJ}$ to a vector $\hat x\in\R^n$ in the same way, using \eqref{eq:Q_R_bounds} we get 
\begin{align*}
    \ell_1(x_\II, \hat x_\JJ) & = x^TQ \hat x + \hat x^T Q x \geq - \| x^TQ \hat x  \| - \| \hat x^T Q x \|\\ 
    & \ge -2M_Q\|x\|\,\|\hat x\| = -2M_Q\|x_\II\|\,\|\hat x_\JJ\|.
\end{align*}
Together with \eqref{eq:intell12} this yields 
\begin{align}
\begin{split} \label{eq:intquad}
    & \int_0^\infty - 2M_Q \|\tilde x_\II^*(t)\|\,\|\hat x_\JJ(t)\| + \mu\|\tilde x_\II^*(t)\|^2 dt \\
    \le & \int_0^\infty \ell_1(\tilde x_\II^*(t), \hat x_\JJ(t)) + \ell_2(\tilde x_\II^*(t)) dt \le 0.
\end{split}
\end{align}
Now Young's inequality states that $\|\tilde x_\II^*(t)\|\,\|\hat x_\JJ(t)\|\le \gamma \|\tilde x_\II^*(t)\|^2/2 + \|\hat x_\JJ(t)\|^2/(2\gamma)$ for each $\gamma>0$ and $t \geq 0$. From \eqref{eq:intquad} we can thus conclude
\begin{equation*}
    \int_0^\infty -M_Q \gamma\|\tilde x_\II^*(t)\|^2 -M_Q \|\hat x_\JJ(t)\|^2/\gamma + \mu\|\tilde x_\II^*(t)\|^2 dt \le 0, 
\end{equation*}
which yields 
\begin{equation*}
    (\mu-M_Q \gamma)\|\tilde x^*_\II\|_{L_2}^2 \le \frac{M_Q}{\gamma}\| \hat x_\JJ\|_{L_2}^2.
\end{equation*}
Choosing $\gamma = \mu (2M_Q)^{-1}$ this implies 
\begin{equation*}
    \|\tilde x^*_\II\|_{L_2}^2 \le \frac{4 M_Q^2}{\mu^2}\| \hat x_\JJ\|_{L_2}^2,
\end{equation*}
which yields the assertion after taking the square root. 
\hfill\qed
\end{pf}

In the next subsection, we apply Lemma \ref{lemma:pertprop} in the following setting. We first solve the overall OCP \eqref{eq:sysi} - \eqref{eq:J} and then choose $\hat x_j$ as the corresponding optimal trajectories, that is, $\hat x_j = x^*_j$ for $j \in \JJ$. The lemma below states that, under the assumption that every connection originating from a node in $\II$ leads to a node contained in $\II \cup \JJ$, the optimal solution $\tilde x_\II^*$ of the reduced OCP \eqref{eq:ocp_II} coincides\footnote{In case the optimal solutions are not unique, ``coincide'' is to be understood in the following sense: for each optimal solution of the overall problem there exists a solution of the reduced problem whose components for $i\in\II$ coincide, and vice versa.} 
with $x_\II^*$, i.e., with the components of the overall optimal solution belonging to the nodes in $\II$.

\begin{lem}
Consider an OCP of the form \eqref{eq:sysi}-\eqref{eq:J} with initial value $x_0 \in \Omega$. Denote the optimal solution starting at $x_0$ by $x^*$. Further, let $\II$, $\JJ \subseteq \VV$, $\II \cap \JJ = \emptyset$, be two disjoint index sets which satisfy
\begin{equation}
    \forall j \in \VV \colon \distG(j,\II) = 1 \implies j \in \JJ. 
    \label{eq:II_JJ_connection}
\end{equation}
For $j \in \JJ$ set  $\hat x_j = x^*_j$. Then $x_\II^*$ is an optimal solution of the reduced OCP \eqref{eq:ocp_II} with initial value $x_{0, \II}$.  
\label{lem:optpert}
\end{lem}

\begin{pf}
    From \eqref{eq:II_JJ_connection} we obtain that for any node $k \notin (\II \cup \JJ)$ it holds that $\distG(k,I) \geq 2$, which, by the definition of $\GG$ implies $Q_{i,k} = Q_{k,i} = 0$ for all $i \in \II$. Thus, the cost function $\ell_\II$ in \eqref{eq:ocp_II} contains all terms of $\ell$ in \eqref{eq:ell} that involve nodes $i \in \II$. Hence, the remaining terms in $\ell$ depend only on nodes in $\VV \setminus \II$ and are thus not affected by the minimization of $J_\II$. As a consequence, if there existed a solution with a lower value of $J_\II$ than $x_\II^*$, then combining this solution with $x^*_{\GG \setminus \II}$---which by the decoupled dynamics \eqref{eq:sysi} yields a feasible trajectory of the overall system---would result in a solution with a lower value of $J$ than $x^*$. This would contradict the optimality of $x^*$. 
\end{pf}
   
\subsection{Proof of Theorem \ref{thm:main}} \label{subsec:proof_main}
We are now in a position to combine the results from the previous subsections to establish the proof of Theorem \ref{thm:main}.

\begin{pf}[Theorem \ref{thm:main}]
For $k \in \N$ we define the sets 
\begin{equation*}
    V_k:=\{ i\in \VV \,|\, \distG(i,i^*) \ge k\}, 
\end{equation*}
and 
\begin{equation*}
    W_k:=\{ i\in \VV \,|\, \distG(i,i^*) = k\}.
\end{equation*}
We first prove the inequality
\begin{equation}
    \|x^*_{V_{qk}}\|_{L_2} \le \widehat S \hat\rho^k|x_{0,i^*}|,
\label{eq:vk}
\end{equation} 
for $\hat\rho =1/2$, $\widehat S=\max\{C_{\textrm{init}},C_{\textrm{prop}}\}$ and $q$ from the statement of the theorem by induction over $k$. 

For $k=0$ we can use Lemma \ref{lemma:inipert} to obtain the inequality
\[  \|x^*_{V_0}\|_{L_2} = \|x^*\|_{L_2} \le C_{\textrm{init}} |x_{0,i^*}|, \]
which proves \eqref{eq:vk} for $k=0$ since $\widehat S\ge C_{\textrm{init}}$.

For the induction step $k\to k+1$, the induction assumption is that \eqref{eq:vk} holds for some $k$. Then we obtain
\begin{equation}
    \sum_{k'\ge qk} \|x^*_{W_{k'}}\|_{L_2}^2 = \|x^*_{V_{qk}}\|_{L_2}^2 \le \big(\widehat S \hat \rho^k|x_{0,i^*}|\big)^2.
    \label{eq:indasm}
\end{equation}
This implies that there is $k^*\in\{qk,qk+1,\ldots,qk+q-1\}$ such that
\begin{equation}
    \|x^*_{W_{k^*}}\|_{L_2} \le \hat\rho^{k+1}|x_{0,i^*}|.
    \label{eq:W_k_star}
\end{equation} 
This holds, because otherwise 
\begin{align*}
    \sum_{k' = qk}^{q(k+1)-1} \|x^*_{W_{k'}}\|_{L_2}^2 > q \big(\hat \rho^{k+1}|x_{0,i^*}|\big)^2 \geq \big(\widehat S \hat \rho^k|x_{0,i^*}|\big)^2,
\end{align*}
which contradicts \eqref{eq:indasm}. 

We now choose $\JJ=W_{k^*}$, $\II=V_{k^*+1}$, and set $\hat x_j=x^*_j$ for all $j\in\JJ$. Note that this choice of $\II$ and $\JJ$ satisfies the condition \eqref{eq:II_JJ_connection} of Lemma \ref{lem:optpert}. Hence, applying Lemma \ref{lemma:pertprop} and Lemma \ref{lem:optpert} together with \eqref{eq:W_k_star} yields
\begin{align*}
    \|x^*_{V_{k^*+1}}\|_{L_2} = & \|x^*_\II\|_{L_2} \le C_{\textrm{prop}} \| \hat x_\JJ\|_{L_2} \\
    = & C_{\textrm{prop}} \| x^*_{W_{k^*}}\|_{L_2} \le C_{\textrm{prop}} \rho^{k+1}|x_{0,i^*}|.
\end{align*}

Since $k^*\le qk+q-1$ and the $V_k$ are decreasing with increasing $k$, the same inequality holds for $\|x^*_{V_{qk+q}}\|_{L_2}$. As $qk+q = q(k+1)$, this yields \eqref{eq:vk} for $k+1$ and thus proves the induction step.

Now \eqref{eq:vk} implies
\[ \|x_{V_{p}}^*\|_{L_2} \le \widehat S \hat\rho^{p/q}|x_{0,i^*}|  \]
for any $p \in \N$ being a multiple of $q$. For each such $p$ and all $k\in \{p+1,\ldots,p+q-1\}$ we obtain
\begin{align*}
    \|x_{V_{k}}^*\|_{L_2} \le & \|x_{V_{p}}^*\|_{L_2}  
    \le \widehat S \hat\rho^{p/q}|x_{0,i^*}| = \widehat S \hat\rho^{k/q}\hat\rho^{(p-k)/q}|x_{0,i^*}| 
    \\ \le & 2 \widehat S \rho^k |x_{0,i^*}| 
    = S \rho^k|x_{0,i^*}|.
\end{align*}
This shows the assertion, since for any $\WW \subseteq \VV$ defining $k \coloneqq \distG(i^*,\WW)$ yields $\WW \subseteq V_k$ and thus $\|x_\WW^*\|_{L_2} \le \|x^*_{V_{k}}\|_{L_2}$. 
\end{pf}

\section{Numerical Experiment} \label{sec:numerics}

In this section, we numerically illustrate the theoretical findings of Theorem~\ref{thm:main}. We consider a group of $s \in \N$ vehicles. Each vehicle is modeled as a two-dimensional system with state $x_i = (y_i, v_i) \in \R^2$, $i \in \VV$, where $y_i$ and $v_i$ denote the position and velocity of the $i$-th vehicle, respectively. The resulting state dimension is therefore $n = 2s$. 

The dynamics of each vehicle are based on a double-integrator model, augmented by viscous and quadratic drag terms, and are given by
\begin{equation*}
    \dot x_i = f_i(y_i, u_i) = \begin{bmatrix}
        v_i \\ 
        - \beta v_i - \kappa v_i | v_i| + u_i,
    \end{bmatrix}
\end{equation*}
where $\beta, \kappa \in \R_{\ge 0}$ are fixed parameters and $u_i$ denotes the control input of vehicle $i$. Such quadratic drag terms have, for instance, been used in \cite{ray2015analytic} and in Example~7.1 of \cite{slotine1991applied} to model motions in fluids.

The control objective is to stabilize a chain of vehicles at the origin while penalizing deviations between neighboring vehicles. This reflects the goal of bringing all vehicles to rest while maintaining coherence of the formation and leads to the cost functional
\begin{align*}
    	J(x,u) 
    = \int_0^\infty & \sum_{i = 1}^{s} y_i^2  + \gamma v_i^2 
    + \sum_{i = 1}^{s-1} (y_{i+1}-y_i)^2 
    + \delta \lVert u\rVert_2^2 \, dt,
\end{align*}
where $\gamma, \delta \in \R_{> 0}$. We can write $J$ in the form \eqref{eq:J} with the cost function $\ell$ as in \eqref{eq:ell} by setting $R = \delta I_s$ and 
\begin{equation*}
    Q = \begin{bmatrix}
		3 & 0 & -1 & 0 & \dots \\ 
		  0 & \gamma  & 0 & 0 & \dots & \\ 
		-1 & 0 & 3 & 0  &  -1  & \\ 
        0 & 0 & 0 & \gamma & 0 & \\
		\vdots  & \vdots &  &  & \ddots & 
    \end{bmatrix} - e_1 e_1^T - e_{n-1} e_{n-1}^T \in \R^{n \times n}, 
\end{equation*}
where $e_i \in \Rn$ denotes the $i$-th canonical unit vector. 

The corresponding graph $\GG$, induced by the block-tridiagonal structure of $Q$, contains one node for each vehicle and connects each vehicle to its immediate predecessor and successor; see Figure~\ref{fig:convoy_graph}.

\begin{figure}[h]
\begin{center}
\includegraphics[width=8cm]{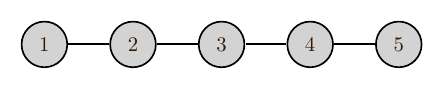}   
\caption{Graph for $s=5$ vehicles.} 
\label{fig:convoy_graph}
\end{center}
\end{figure}
Note that, overall, this yields an OCP of the form \eqref{eq:sysi}-\eqref{eq:J} with decoupled dynamics and coupling introduced through the matrix $Q$. Moreover, it can be shown that Assumption \ref{asm:expstab} is satisfied for a bounded domain $\Omega \subset \R^{n}$ of initial values. By choosing feedback functions $u_i = F_i(x_i) = - y_i + \kappa v_i |v_i| $ we obtain the linear closed-loop system 
\begin{equation*}
    \dot x_i = \begin{pmatrix}
        0 & 1 \\ 
        -1 & -\beta
    \end{pmatrix} x_i
\end{equation*}
for each vehicle $i \in \VV$, which is exponentially stable for any $\beta > 0$. Thus, there exist $C, \sigma > 0$ such that for all $x_{0,i} \in \R^2$ we have $ \| x_i(t, x_{0,i}, u_i) \| \leq C e^{- \sigma t} \| x_{0,i} \|$. It follows that 
\begin{align*}
    | u_i(t) | & = \big| - y_i(t) + \kappa |v_i(t)| v_i(t) \big| \\
    & \leq (C + \kappa C^2 \sup_{x \in \Omega} \| x \|) e^{- \sigma t} \| x_{0,i} \|, 
\end{align*}
which verifies Assumption~\ref{asm:expstab} by combining the component-wise estimates to obtain bounds for $x(t)$ and $u(t)$.

Now for some fixed $i^* \in \VV$ we specify an initial condition $x_0 \in \Rn$ such that
\begin{align*}
    x_{0,i^*} = \begin{bmatrix} 1 \\ 1 \end{bmatrix} \in \R^2,
    \qquad
    x_{0,i} = 0 \in \R^2 \quad \text{for all } i \neq i^*.
\end{align*}
We then employ a model predictive control scheme to numerically approximate the optimal trajectory. To this end, we choose a fixed step size $h > 0$ and perform $N \in \N$ MPC steps until the approximated optimal trajectory satisfies $\|x^*(hN)\| \leq \varepsilon$ for a fixed threshold $\varepsilon > 0$. We then approximate
\begin{align*}
\|x^*\|_{L_2}^2
&= \int_0^\infty \|x^*(t)\|^2 dt
\approx h \sum_{k=1}^{N} \|x^*(hk)\|^2 .
\end{align*}
The nonlinear MPC controller is implemented using \texttt{do-mpc} (see \cite{fiedler2023dompc}). We consider the case of $s = 25$ vehicles, i.e., $n = 50$, with model parameters 
$\beta = 5$ and $\kappa = 10$, regularization parameters $\gamma = \delta = 10^{-1}$, and choose $i^* = 12$. Using a step size of $h = 5 \times 10^{-2}$ and a tolerance  $\varepsilon = 10^{-4}$, the algorithm terminates after $317$ iterations. 
Figure~\ref{fig:decay_single} shows the ${L}_2$-norms of the optimal state trajectories $x_i^*$ for all vehicles $i \in \VV$ on a logarithmic scale, for $i^* = 12$ as well as for the cases $i^* \in \{1,25\}$. The results clearly demonstrate an  exponential decay of the trajectory norms with respect to the graph distance, in agreement with Theorem~\ref{thm:main}.

\begin{figure}[h]
\begin{center}
\includegraphics[width=8.6cm]{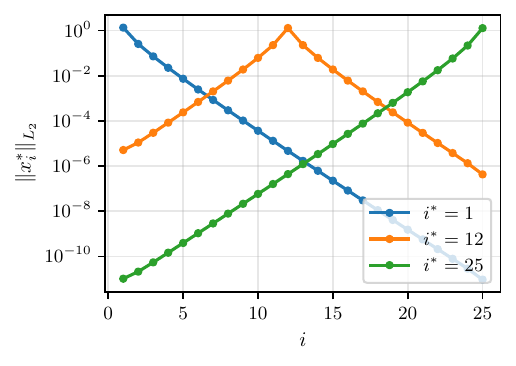}
\vspace{-0.5cm}
\caption{Decay of the ${L}_2$-norms of the optimal state trajectories $x^*_i$ for $i^*\in\{1,12,25\}$.} 
\label{fig:decay_single}
\end{center}
\end{figure}

\section{Conclusion} \label{sec:conclusion}

In this paper, we studied spatial decay properties of sensitivities in nonlinear optimal control problems with decoupled dynamics and quadratic cost. Under the assumption of exponential null-controllability, we showed that a perturbation of the zero initial condition at a single node leads to perturbations in the components of the optimal trajectory that decay exponentially with respect to graph distance from the perturbed node. For future research, it is of interest to extend the result to nonzero reference trajectories, more general classes of nonlinear dynamics with coupling between subsystems, as well as to broader classes of cost functionals.

\bibliography{literature}                                                           

@Article{shin2022lqr,
  author   = {Shin, Sungho and Lin, Yiheng and Qu, Guannan and Wierman, Adam and Anitescu, Mihai},
  journal  = {SIAM J. Control Optim.},
  title    = {Near-Optimal Distributed Linear-Quadratic Regulator for Networked Systems},
  year     = {2023},
  number   = {3},
  pages    = {1113-1135},
  volume   = {61},
}

@Article{shin2021controllability,
 author    = {Shin, Sungho and Zavala, Victor M},
 journal   = {IFAC-PapersOnLine},
 title     = {Controllability and observability imply exponential decay of sensitivity in dynamic optimization},
 volume = {54},
 number = {6},
 pages = {179-184},
 year = {2021},
}

@Article{shin2022expDecayNLP,
  author   = {Shin, Sungho and Anitescu, Mihai and Zavala, Victor M.},
  journal  = {SIAM J. Optim.},
  title    = {Exponential Decay of Sensitivity in Graph-Structured Nonlinear Programs},
  year     = {2022},
  number   = {2},
  pages    = {1156-1183},
  volume   = {32},
}

@inproceedings{zhang2023optimal,
  title={On the optimal control of network {LQR} with spatially-exponential decaying structure},
  author={Zhang, Runyu Cathy and Li, Weiyu and Li, Na},
  booktitle={2023 American Control Conference (ACC)},
  pages={1775--1780},
  year={2023},
  organization={IEEE}
}

@article{na2020exponential,
  title={Exponential decay in the sensitivity analysis of nonlinear dynamic programming},
  author={Na, Sen and Anitescu, Mihai},
  journal={SIAM J. Optim.},
  volume={30},
  number={2},
  pages={1527--1554},
  year={2020},
  publisher={SIAM}
}

@inproceedings{sperl2023separable,
  title={Separable approximations of optimal value functions under a decaying sensitivity assumption},
  author={Sperl, Mario and Saluzzi, Luca and Gr{\"u}ne, Lars and Kalise, Dante},
  booktitle={2023 62nd IEEE Conference on Decision and Control (CDC)},
  pages={259--264},
  year={2023},
  organization={IEEE}
}

@book{Bellman1957,
  author    = {Richard Bellman},
  title     = {Dynamic Programming},
  publisher = {Princeton University Press},
  year      = {1957},
  address   = {Princeton, NJ}
}

@article{qu2022scalableReinforcement,
author = {Qu, Guannan and Wierman, Adam and Li, Na},
title = {Scalable Reinforcement Learning for Multiagent Networked Systems},
year = {2022},
issue_date = {November-December 2022},
publisher = {INFORMS},
address = {Linthicum, MD, USA},
volume = {70},
number = {6},
issn = {0030-364X},
journal = {Oper. Res.},
pages = {3601–3628},
numpages = {28}
}

@Article{GrSS20,
  author    = {L. Gr\"une and M. Schaller and A. Schiela},
  journal   = {J. Differential Equations},
  title     = {Exponential sensitivity and turnpike analysis for linear quadratic optimal control of general evolution equations},
  year      = {2020},
  number    = {12},
  pages     = {7311--7341},
  volume    = {268}
}

@article{goettlich2025perturbationspdeconstrainedoptimalcontrol,
      title={Perturbations in {PDE}-constrained optimal control decay exponentially in space}, 
      author={Simone Göttlich and Manuel Schaller and Karl Worthmann},
    journal={ESAIM Control Optim. Calc. Var.},
    volume={31},
    number={27},
    year={2025},
    publisher={EDP Sciences}
}

@article{goettlich2026spatialexponentialdecayperturbations,
	author = {{G\"ottlich, Simone} and {Oppeneiger, Benedikt} and {Schaller, Manuel} and {Worthmann, Karl}},
	title = {Spatial exponential decay of perturbations in optimal control of general evolution equations},
	journal = {ESAIM Control Optim. Calc. Var.},
	year = {2026},
	volume = {32},
	number = {11},
}

@Article{sperl2026separable,
  author  = {Sperl, Mario and Saluzzi, Luca and Kalise, Dante and Gr\"{u}ne, Lars},
  journal = {SIAM J. Control Optim.},
  title   = {Separable approximations of optimal value functions and their representation by neural networks},
  year    = {2026},
  number  = {3},
  pages   = {1099-1126},
  volume  = {64}
}

@article{fiedler2023dompc,
  author  = {Fiedler, F. and Karg, B. and Lüken, L. and Brandner, D. and Heinlein, M. and Brabender, F. and Lucia, S.},
  title   = {do-mpc: Towards FAIR nonlinear and robust model predictive control},
  journal = {Control Engineering Practice},
  volume  = {140},
  pages   = {105676},
  year    = {2023},
}

@article{oppeneiger2025spatialdecayperturbationshyperbolic,
title = {Spatial decay of perturbations in transport equations with optimal boundary control},
journal = {IFAC-PapersOnLine},
volume = {59},
number = {8},
author = {Benedikt Oppeneiger and Manuel Schaller and Karl Worthmann},
pages = {66-71},
year = {2025},
}

@book{slotine1991applied,
  title={Applied nonlinear control},
  author={Slotine, Jean-Jacques E and Li, Weiping},
  address = {Englewood Cliffs, NJ}, 
  year={1991},
  publisher={Prentice-Hall}
}

@article{ray2015analytic,
  title={An analytic solution to the equations of the motion of a point mass with quadratic resistance and generalizations},
  author={Ray, Shouryya and Fr{\"o}hlich, Jochen},
  journal={Archive of Applied Mechanics},
  volume={85},
  number={4},
  pages={395--414},
  year={2015},
  publisher={Springer}
}
\appendix
\end{document}